\documentclass[12pt]{amsart}

\usepackage{enumerate, amsmath, amsthm, amsfonts, amssymb, xy,  combelow, mathrsfs, graphicx, paralist, fancyvrb, braket, xcolor, verbatim, tikz, tikz-cd, mathtools,enumitem, comment}
\usepackage[bookmarks, colorlinks=true, linkcolor=blue, citecolor=blue, urlcolor=blue]{hyperref}

\input xy
\xyoption{all}

\DeclareMathOperator{\mor}{\mathsf{Mor}}
\DeclareMathOperator{\quot}{\mathsf{Quot}}

\renewcommand{\tilde}[1]{\widetilde{#1}}

\newtheorem{theorem}{Theorem}
\newtheorem{lemma}{Lemma}
\newtheorem{corollary}{Corollary}
\newtheorem{proposition}{Proposition}
\newtheorem{conjecture}{Conjecture}

\newtheorem{question}{Question}
\newtheorem*{theorem*}{Theorem}
\usepackage{comment}

\theoremstyle{definition}
\newtheorem{remark}{Remark}

\newtheoremstyle{TheoremNum}
        {7pt}{7pt}              %%% space between body and thm
        {\itshape}                      %%% Thm body font
        {}                              %%% Indent amount (empty = no indent)
        {\bfseries}                     %%% Thm head font
        {.}                             %%% Punctuation after thm head
        { }                             %%% Space after thm head
        {\thmname{#1}\thmnote{ \bfseries #3}}%%% Thm head spec
    \theoremstyle{TheoremNum}

\newcommand{\BC}{\mathbb{C}}
\newcommand{\BE}{\mathbb{E}}

\newcommand{\BP}{\mathbb{P}}
\newcommand{\BQ}{\mathbb{Q}}
\newcommand{\BZ}{\mathbb{Z}}

\newcommand{\CE}{\mathcal{E}}
\newcommand{\CF}{\mathcal{F}}

\newcommand{\CO}{\mathcal{O}}

\newcommand{\git}{\mathbin{
		\mathchoice{/\mkern-6mu/}% \displaystyle
		{/\mkern-6mu/}% \textstyle
		{/\mkern-5mu/}% \scriptstyle
		{/\mkern-5mu/}}}% \scriptscriptstyle
\newcommand{\GL}{\mathrm{GL}}

\begin{document}

\parindent=30pt

\baselineskip=17pt

\title [A short way of counting maps]{A short way of counting maps to hypersurfaces in Grassmannians}

\vskip.2in

\author{Alina Marian}
	\address{The Abdus Salam International Centre for Theoretical Physics, Strada Costiera 11, 34151 Trieste \newline \text{ } \qquad Department of Mathematics, Northeastern University, 360 Huntington Avenue, Boston, MA 02115}
	\email{amarian@ictp.it}

\vskip.2in    
    \author{Shubham Sinha}
	\address{The Abdus Salam International Centre for Theoretical Physics, Strada Costiera 11, 34151 Trieste}
	\email{ssinha1@ictp.it}
\date{}
\vskip.2in

\begin{abstract}
Using a Quot scheme compactification, we calculate the virtual count of maps of degree $d$ from a smooth curve of genus $g$  to a hypersurface in a Grassmannian, sending specified points of the curve to special Schubert subvarieties restricted to the hypersurface. We study the question of whether this virtual count is in fact enumerative under suitable conditions on the hypersurface, in the regime when the map degree $d$ is large.

\end{abstract}

\maketitle

\section{Introduction}

Let $C$ be a smooth complex projective curve of genus $g$. Consider the Grassmannian $G(r, n)$ of $r$ planes in $\mathbb C^n$ with tautological sequence
$$0 \to S \to \CO^n \to Q \to 0.$$
Let $$X_{\ell} = Z (s) \subset G (r, n)$$ be a hypersurface cut out by a general section 
\begin{equation}
\label{eqn:s on grass}
s \in H^0 \left ( G (r, n), \, \left (\det S^\vee \right )^{\otimes \ell}\right ).
\end{equation}
In this note we address the problem of enumerating maps $f: C \to X_{\ell}$ subject to incidence conditions with special Schubert subvarieties $\sigma_k, 1 \leq k \leq r,$ at fixed domain points. We will use the Quot compactification $\quot_d(C, X_{\ell})$ of the space $\mor_d(C, X_{\ell})$ of degree $d$ maps from $C$ to $X_\ell$, defined as follows. To start, let $\quot_d(C, G(r, n))$ be the Quot scheme parametrizing short exact sequences 
\begin{equation}
\label{eqn:quotpoint}
0\to E \to \CO_C^{n}\to F\to 0
\end{equation}
on $C$ where $E$ has rank $r$ and degree $-d$. The space $$\mor_d(C, G(r, n)) \subset \quot_d(C, G(r, n))$$ is the open subscheme of exact sequences \eqref{eqn:quotpoint} with locally free quotients. Let 
\begin{equation}
\label{eqn:universalseq}
0 \to \CE \to \CO^n \to \CF \to 0 \, \, \, \text{on} \, \, \, \quot_d(C, G(r, n)) \times C
\end{equation}
be the universal sequence. The section \eqref{eqn:s on grass} can be viewed as a general element of $$\text{Sym}^{\ell}\, V^{\vee}\simeq H^0 (\BP (V), \CO (\ell)), \, \, \text{where} \,\, V = \wedge^r \BC^n.$$
$X_\ell$ is correspondingly the intersection of a general degree $\ell$ hypersurface in $\mathbb P (V)$ with $G(r, n),$ viewed under the Pl\"{u}cker embedding $G(r, n) \hookrightarrow \mathbb P (V).$
Through the universal sequence \eqref{eqn:universalseq}, $s$ induces then a section 
\begin{equation}
\label{eqn: s on quot}
 \CO \xrightarrow{s}  \text{Sym}^\ell \, V^\vee \otimes \CO\to \left (\det \CE^\vee \right )^{\otimes \ell} \, \, \text{on} \, \, \quot_d(C, G(r, n)) \times C
 \end{equation}
of the line bundle $\left (\det \CE^\vee \right )^{\otimes \ell}.$ We define $$\quot_d(C, X_{\ell})\subset \quot_d(C, G (r, n))$$ as the scheme theoretic locus of subsheaves $\{E \subset \CO^n \} 
\in \quot_d (C, G (r, n))$ for which the section \eqref{eqn: s on quot} vanishes at all points of $C.$ By definition, 
\begin{equation}
\mor_d(C, X_{\ell}) \subset \quot_d(C, X_{\ell}). 
\end{equation}
We denote as $\pi, \, \rho$ the projections from the product $\quot_d(C, G(r, n)) \times C$ to the two factors. Whenever 
$$d \ell > 2g-2,$$ the pushforward
\begin{equation}
\BE_{\,\ell} = \pi_* \left ( (\det \CE^\vee)^{\otimes \ell} \right ) \, \, \, \text{on} \, \, \, \quot_d (C, G(r, n))
\end{equation}
is locally free of rank $d \ell - g + 1.$ 
In this case, we have scheme-theoretically
\begin{equation}
\label{eqn:basicinclusion}
\quot_d(C, X_{\ell}) = Z (\tilde{s}) \xhookrightarrow{\iota} \quot_d(C, G (r, n)),
\end{equation}
where  
\begin{equation}
\label{eqn:tildes}
\tilde s \in H^0 \left (\quot_d(C, G(r, n)), \, \BE_{\,\ell} \right )
\end{equation}
is the image of the section \eqref{eqn: s on quot} under $\pi_*$. 

As we will recall in the next section, both $\quot_d(C, G(r, n))$ and, more subtly, $\quot_d(C, X_{\ell})$ are known to admit perfect obstruction theories and virtual classes which are naturally compatible under the inclusion \eqref{eqn:basicinclusion}, satisfying
\begin{equation}\label{eq:Compatibility_virtual_class}
\iota_*[\quot_d(C, X_\ell)]^{vir} = c_{top}(\BE_{\, \ell} ) \cap [\quot_d(C, G(r,n))]^{vir}
\end{equation}
for all $d\ell >2g-2 $.

\vskip.2in

Over the Quot scheme $\quot_d(C, G(r, n)),$ compactifying the space of degree $d$ maps to the Grassmannian $G(r, n)$ itself, an interesting class of virtual intersections is calculated by the Vafa-Intriligator formula. Indeed, for a point $p \in C$, we let $\CE_{p}$ denote the restriction of the rank $r$ universal subsheaf $\CE$ to $\quot_d(C, G(r, n)) \times \{p\}$, and focus on its Chern classes 
\begin{equation}\label{eq:a_i}
    a_i := c_i(\CE_{p}^\vee), \, \, \, 1\leq i \leq r.
\end{equation}
The formula calculates the virtual top intersections of these classes. Set $$ e = d n +r(n-r)(1-g) = \, \text{virtual dim of} \,  \quot_d(C,G(r,n)).$$

\begin{theorem*}[Vafa-Intriligator Formula, cf. \cite{mo,bertram,st}]\label{thm:VI_formula}
	For any monomial \\ $P = \prod_{k=1}^{t}a_{i_k}$, $1\le i_k\le r$ of weighted degree $e=d n +r(n-r)(1-g)$, we have
\begin{equation}
\label{eqn:VI}
    \int_{[\quot_d(C,G(r,n))]^{vir}}P= (-1)^{d(r-1)}\cdot \sum_{\zeta_1,\dots,\zeta_r}\prod_{k=1}^{t} e_{i_k}(\zeta_1,\dots,\zeta_r)J^{g-1}(\zeta_1\dots\zeta_r),
\end{equation}
	with the sum being taken over all $\binom{n}{r}$ $r$-tuples $(\zeta_1,\dots,\zeta_r)$ of distinct $n^{th}$ roots of unity. Here, $e_{i_k}$ is the $i_k$th elementary symmetric polynomial in $r$ variables, and $$ J(x_1,\dots,x_r)= \prod_{i=1}^{r}nx_i^{n-1}\prod_{1\le i\ne j\le r}(x_i-x_j)^{-1}.$$
\end{theorem*}
When the degree $d$ is sufficiently large with respect to the genus $g$ of $C$, Bertram \cite{bertram} showed that the above integral is enumerative and counts the number of maps from $C$ to $G(r,n)$ sending $t$ distinct points $p_1, \dots, p_t$ to special Schubert subvarieties of type $\sigma_{i_k}=c_{i_k}(S^{\vee})$, in general position. 

%This enumeration problem was also solved by Siebert and Tian \cite{st} %using the stable map compactification.

\vskip.2in

We now turn to intersections of type $a$ in the subscheme
$\quot_d(C, X_{\ell})$ of virtual dimension $e_{\ell}=e-(d\ell -g+1),$ and establish the following count. The analogous virtual count of maps from $C$ to a complete intersection in $G(r,n)$ is addressed in Theorem~\ref{Thm:VI_formula_complete_intersection} of Section~\ref{sec:complete_intersection}.  

\begin{theorem}\label{thm:VI-type_hypersurface}
	Assume $d \ell >2g-2$. Let $P = \prod_{k=1}^t a_{i_k}, \, \, 1\leq i_k \leq r$ be a monomial of weighted degree $e_{\ell}=e-(d\ell -g+1).$ Then
	\begin{align}\label{eqn:VI_hypersurface}
	\int_{[\quot_d(C,X_{\ell})]^{vir}}^{}P=
    %T_{d,g}(\ell)
    \frac{(n-\ell)^g\cdot \ell^{\,d\ell-g+1}}{n^g}\int_{[\quot_d(C,G(r,n))]^{vir}}^{}a_1^{d\ell-g+1}P.
	\end{align}
	%where
	%\[T_{d,g}(\ell)=\ell^{d\ell-g+1}\bigg(\sum_{i=0}^{d}(-1)^i\binom{g}{i}\ell^iN^{-i} \bigg). \]
    %$    T_{d,g}(\ell)= \ell^{d\ell-g+1}(1-\ell/n)^g.$
	%Note that when $d\ge g$, $T_{d,g}(\ell)= \ell^{d\ell-g+1}(1-\ell/N)^g$. \textcolor{red}{check!}
The integral on the right hand side admits a closed-form evaluation by the Vafa-Intriligator Formula \eqref{eqn:VI}. The intersection product \eqref{eqn:VI_hypersurface} is the virtual count of maps from $C$ to $X_\ell$ sending points $p_k \in C$ to $Y_{i_k} \cap X_\ell$, $1 \leq k \leq t$, where $Y_{i_k}\subset G(r,n)$ is a special Schubert subvariety of type $\sigma_{i_k}.$
\end{theorem}

%The following are some specialization and consequences of %Theorem~\ref{thm:VI-type_hypersurface}. 
% When $X_\ell\subset \BP^{r} = G(r,r+1)$ is a degree $\ell$ hypersurface
%\begin{align}\label{eq:Projective_hypersurface_virtual}
 %   \int_{[\quot_d(C,X_\ell)]^{vir}}\prod_{k=1}^t a_{i_k} = \ell^{\ell %d- g+1}(r+1-\ell)^{g},
%\end{align}
%where $1\le i_{k}\le r$, $i_1+\cdots + i_t$ equals the expected %dimension $e_{\ell}=(r+1-\ell)d-(r-1)(g-1)$ and $d>2g-2$.
%According to Theorem~\ref{thm:enumerative}, the virtual invariants for %the hypersurface $X_{\ell} \subset \BP^r$ in %\eqref{eq:Projective_hypersurface_virtual} are enumerative when $i_k < %n - \ell$.

\begin{remark}
For hypersurfaces $X_{\ell}\subset \BP^r$, by specializing \eqref{eqn:VI_hypersurface}, the virtual count of maps from $C$ to $X_{\ell}$ sending $t=e_\ell$ points to $t$ hyperplanes in $X_{\ell}$, is obtained:
 \begin{equation}\label{eq:intro_VI_hypersurface_P^r}
      \int_{[\quot_d(C,\BP^r)]^{vir}} a_{1}^{e_{\ell}} = \ell^{d\ell - g + 1} (r + 1 - \ell)^g.
 \end{equation}
 Here the expected dimension is given by $e_{\ell} = d(r+1 - \ell) + (1 - g)(r - 1)$.
 \end{remark}
 \begin{remark}
 The Lagrangian Grassmannian $LG(2,4)\subset G(2,4),$ parametrizing Lagrangian subspaces in a symplectic vector space $\BC^{4},$ is cut out by a section of $\det S^\vee$. Through Theorem \ref{thm:VI-type_hypersurface}, the virtual intersection numbers have the simple formulas 
\[
\int_{[\quot_d(C,LG(2,4))]^{vir}} a_{1}^{m_1}a_2^{m_2} = 2^{2d-m_2-g+1} \cdot 3^g,
\]
where $m_1+2m_2=3(d-g+1)$ and $d>2g-2$.

More generally, the symplectic Grassmannian $SG(2,n) \subset G(2,n),$ parametrizing rank two isotropic subspaces in a symplectic vector space $\BC^{n}$, is cut out by a section of $\det S^\vee$. In this case, $\quot_d(C,SG(2,n))$ is precisely the Isotropic Quot scheme considered in \cite{si}, which parametrizes rank two isotropic subsheaves of the trivial bundle $\CO^{n}$ equipped with a symplectic form. Our method provides a short proof of \cite[Theorem~1.3]{si} in the setting $d > 2g - 2$.
\end{remark}

\begin{remark}
The proof of Theorem~\ref{thm:VI_formula} uses the virtual class compatibility \eqref{eq:Compatibility_virtual_class} along with the explicit evaluation of a large class of virtual top intersections on $\quot_d(C,G(r,n))$ extending beyond the Vafa-Intriligator formula \eqref{eqn:VI}. This proof will be given in Section~\ref{sec:proof_of_VI_formula}.
\end{remark}

\begin{remark}\label{rem:Tevelev_degree}
Theorem \ref{thm:VI-type_hypersurface} fits in the broad study of quasimap invariants by Ciocan-Fontanine and Kim (cf. \cite{cfk1, cfk2, cfk3}, see also \cite{mop}), in the simpler setting of a fixed domain curve. Related fixed-domain virtual counts were explored in a stable map setting in recent papers on Tevelev degrees (cf. \cite{{bp}, {cela}, {celalian},{fl}}), as follows. Consider the moduli space $\overline{\mathcal{M}}_{g,t}(X, \beta)$ of stable maps to a smooth projective variety $X$ and the morphism
$
\tau:\overline{\mathcal{M}}_{g,t}(X, \beta) \to \overline{\mathcal{M}}_{g,t} \times X^t,
$
remembering the domain and evaluating at the $t$ points.
When the expected fiber dimension of $\tau$ is zero,
the (virtual) Tevelev degree $\mathsf {Tev}_{g, n, \beta}$ is defined by 
$$\tau_* [\overline{\mathcal{M}}_{g,t}(X , \beta)]^{vir} = \mathsf {Tev}_{g, n, \beta}\,  [\overline{\mathcal{M}}_{g,t} \times X^t]$$
and gives the virtual count of maps from a fixed domain curve to $X$ with prescribed outputs at $t$ domain points. When $X$ is a Grassmannian, initial calculations of Tevelev degrees recovered particular cases of the Vafa-Intriligator formula \eqref{eqn:VI}; further computations for Fano hypersurfaces and complete intersections in projective space were carried out in \cite {{bp}, {cela}}.  Theorem \ref{thm:VI-type_hypersurface} calculates virtual map counts to a hypersurface in any Grassmannian, using the Quot compactification. The incidence conditions are at special Schubert subvarieties, not points, in the target. 

\end{remark}

\vskip.1in
We now examine the enumerativity of the virtual invariants of Theorem \ref{thm:VI-type_hypersurface}.
This hinges on the basic question whether the space of maps from $C$ to a general hypersurface in $G(r, n)$ has expected dimension for sufficiently large degree $d$. We expect 
\begin{conjecture}
\label{conjecture:expected dim}
For a general hypersurface $X_{\ell}$ in $G(r,n)$ with $\ell < n,$ there exists a threshold degree $d_0(g, r, n)$ so that the space of maps $\mathrm{Mor}_d(C, X_\ell)$ has expected dimension for all $ d > d_0.$
\end{conjecture}

We will show that if the space of maps $\mor_d (C, X_\ell)$ has expected dimension, the boundary of the Quot compactification $\quot_d (C, X_\ell)$ is well-behaved, allowing for an enumerative interpretation of the virtual type $a$ invariants. 

\vskip.1in

\begin{theorem}
\label{thm:enumerative}
Let $X_\ell\subset G(r,n)$ be a hypersurface such that the space of maps $\mor_{d}(C,X_{\ell})$ has expected dimension for all sufficiently large $d$. Then the virtual type $a$ integral
\[
\int_{[\quot_d(C,X_{\ell})]^{vir}}^{}\prod_{k=1}^t a_{i_k}
\]
in Theorem \ref{thm:VI-type_hypersurface} is enumerative for all $d$ sufficiently large, whenever $i_k < n- \ell,$ with $1\leq k \leq t.$
     The virtual integral counts the actual number of maps from $C$ to $X_\ell$ sending $t$ distinct domain points $p_1, \ldots, p_t$ to the intersection of $X_\ell$ with 
    general Schubert subvarieties of type $c_{i_k}(S^\vee),$ for $ 1\leq k \leq t.$

In the case of a complete intersection $X_{\underline{\ell}}$ in $G(r,n)$, of multidegree $\underline{\ell}=(\ell_1,\dots,\ell_u)$, the analogous enumerativity statement holds subject to the condition $i_k\le n-\sum_{j=1}^{u}\ell_j$ for each $1\le k\le t$. 
\end{theorem}

When $C = \BP^1$, for a general degree $\ell$ hypersurface $X_\ell \subset \mathbb P^r,$ it was shown in \cite{hrs,ry} that whenever $\ell < r-1$, the space of maps $\mor_d (\BP^1, X_\ell)$ is irreducible of the expected dimension for all $d$. Recent developments in arithmetic geometry, beginning with Browning and Vishe \cite{bv}, established irreducibility of $\mor_d(\mathbb{P}^1, X_\ell)$ for arbitrary $X_\ell$ of sufficiently low degree relative to $r$. This was extended to higher genus domain curves in \cite{hase1}. For {\it general} hypersurfaces $X_\ell \subset \mathbb{P}^r$, the bound was significantly improved in \cite[Corollary~7]{hase2} and \cite{saw} to $\ell \le (r+6)/4$.  

In the case of a hypersurface in $\BP^r$, we note the following corollary of Theorems \ref{thm:VI_formula} and \ref{thm:enumerative}, which will be discussed in Section~\ref{sec:hypersurfaces_Pn}.

\begin{corollary}\label{cor:VI_hypersurface_projective_space}
    Let $X_{\ell}\subset \BP^r$ be a hypersurface such that $\mor_d(C,X_{\ell})$ has expected dimension for $d$ sufficiently large. There are exactly $\ell^{d\ell - g + 1} (r + 1 - \ell)^g$ maps from $f:C\to\BP^r$ such that $f(C)\subset X_{\ell}$ and $f$ sends distinct points $p_1,\dots,p_t$ to codimension $i_k$ planes $H_{i_k}\subset \BP^r$ in general position for $1\le k\le t$. Here we assume $\sum_{k=1}^{t} i_k = e_\ell$ and  $1 \le i_k \le r - \ell$ for all $1\le k\le t$.
\end{corollary}

\begin{remark}\label{rem:weak_g-convexity}
   It is convenient to introduce a new notion to describe varieties whose associated space of curves of fixed genus has expected dimension.
 Let $X$ be a smooth projective variety and let $\beta \in H_2(X, \mathbb{Z})$ be an effective curve class. We say that the pair $(X, \beta)$ is \textbf{weakly $g$-convex} if there exists a bound $d_0 = d_0(g, X, \beta)$ such that the space
\[
\mathrm{Mor}_{d\beta}(C, X) = \{ f \colon C \to X \mid f_*[C] = d\beta \}
\]
has the expected dimension for any curve $C$ of genus $g$ and all $d \ge d_0$.

We recall that a variety $X$ is said to be \emph{convex} if $H^1(\mathbb{P}^1, f^*T_X) = 0$ for every morphism $f \colon \mathbb{P}^1 \to X$. In particular, if $X$ is a homogeneous variety $G/P$, then $X$ is convex and $\mor_\beta(C,X)$ is moreover irreducible for all $\beta$ (cf. \cite{kp}). If $X$ is convex, then $(X,\beta)$ is weakly $0$-convex for any $\beta$. It is not known however whether convexity implies weak $g$-convexity for $g > 0$. 

\end{remark}

We summarize below the known examples of weak convexity from the literature, studied by various authors using different techniques. In all the examples listed, except the last, $H_2(X, \mathbb{Z})$ has rank one, and there is a natural choice of the homology class $\beta$.

\begin{itemize}
    \item The Grassmannian $G(r, n)$ is weakly $g$-convex for all $g \ge 0$ \cite{bdw}, proven using the irreducibility of the moduli space of stable bundles on $C$.
    
    \item A general hypersurface $X_\ell \subset \mathbb{P}^r$ is weakly $0$-convex for $\ell < r - 1$ (see \cite{hrs, ry}), proven by analyzing the boundary of the moduli space of stable maps.

    \item A general complete intersection $X_{\underline{\ell}}$ in $\BP^r$ is weakly $0$-convex where multidegree $\underline{\ell}=(\ell_1,\dots,\ell_m)$ satisfy $\sum_{i=1}^{m}\ell_i< 2r/3$, see \cite{bk}.
    
    \item Any smooth hypersurface $X_\ell \subset \mathbb{P}^r$ with small $\ell$ is weakly $g$-convex for all $g \ge 0$ (see \cite{bv} and \cite{hase1} for explicit bounds on $\ell$ in the cases $g = 0$ and arbitrary $g$, respectively); these results are obtained via point counts over finite fields. In the genus zero case, analogous results were proved for any complete intersections in $\BP^r$ of low multidegrees in \cite{bvy}.
    
    \item A general hypersurface $X_\ell \subset \mathbb{P}^r$ is weakly $g$-convex for all $g$ when $5\le \ell \le (n+6)/4$ (see \cite[Corollary~7]{hase2} and \cite{saw}); this is done by explicit points count for Fermat hypersurfaces over finite fields using circle method.
    
    \item The Lagrangian and orthogonal Grassmannians $LG(n, 2n)$ and $OG(n, 2n)$ are weakly $g$-convex for all $g$, as shown in \cite{cch, cch2} by analyzing the geometric properties of isotropic Quot schemes.
    \item Two-step flag varieties $FL(r_1, r_2; n)$ are weakly $g$-convex for all $g \ge 0$ and for specific choices of the curve class $\beta$; see \cite{rs} for conditions ensuring the irreducibility of the associated hyper/nested Quot schemes.
\end{itemize}

\begin{remark}
In a stable map context, the asymptotic enumerativity of Tevelev degrees was studied in 
\cite[Theorem 11]{lp} and \cite{royaetal} for $X_\ell \subset \mathbb P^r$, but not for hypersurfaces in a general Grassmannian. For low map degrees, the Tevelev numbers of $\BP^r$ were studied in \cite{lian}. Furthermore, this question is addressed in quasimap context for blowups of projective spaces in \cite{celalian2}. Recently, the enumerativity of virtual Tevelev degrees was established for positive symplectic manifolds and general almost complex structure on the domain in \cite{cd}.

The question regarding the enumerativity of virtual invariants over the isotropic Quot scheme $\quot_d(C,SG(r,n))$, for sufficiently large $d$ with respect to $g$, $r$ and $n$, was also raised in \cite{si}. While the symplectic Grassmannian is convex, it would be interesting to study its weak $g$-convexity in generality.
\end{remark}

\noindent
 {\bf Acknowledgements.} We thank Alessio Cela, Izzet Coskun, Matthew Hase-Liu, Carl Lian, and Dragos Oprea for useful discussions.

\begin{comment}
\begin{theorem}\cite{lp,royaetal}
    Suppose $X_{\ell}$ is a general smooth hypersurface of degree $\ell$ in $\BP^{r}$. Furthermore suppose $2\ell \le r+2$ and $\ell>2$, then 
    \[
    \mathrm{Tev_{g,d,t}^{X_{\ell}}} = (\ell !)^t\ell^{(d-t)\ell-g+1}(r+1-\ell)^{g}.
    \]
\end{theorem}

\begin{question}
Note that the $\sigma_{r-1}\in H^*(X_{\ell}, \BQ)$ represents the class of $\ell$ points on $X_{\ell}\subset \BP^{r}$. Naively, the virtual count of maps describing $\mathrm{Tev_{g,d,t}^{X_{\ell}}} $ should be the virtual integral of $(a_{r-1}/\ell)^t$ over $\quot_d(C,X_\ell)$. Comparing the virtual and the enumerative count when $2\ell \le r+2$ and $\ell>2$, 
\[
\mathrm{Tev_{g,d,t}^{X_{\ell}}}= \bigg(\frac{\ell !}{\ell^\ell}\bigg)^t\cdot \int_{[\quot_d(C,X_\ell)]^{vir}}\bigg(\frac{a_{r-1}}{\ell}\bigg)^{t}.
\]
    Can we explain the multiplicative factor $(\ell!/\ell^\ell)^t$ in the above equality geometrically?
\end{question} 

\end{comment}

\vskip.3in

\section{Virtual Fundamental Class}

\vskip.2in

The virtual fundamental class for the Quot scheme $\quot_d(C,G(r,n))$ was described in \cite{cfka}, \cite{mo}. In this section, we briefly review the perfect obstruction theory of the hypersurface Quot scheme $\quot_d(C,X_{\ell})$, using the \cite{ckm} framework of quasimaps to GIT quotients. Readers interested in explicitly constructing the virtual fundamental class for $\quot_d(C,X_{\ell})$ can also follow the construction in \cite[Section 2]{si}, making the necessary adjustments.

\subsection{Quasimap Space to $X_\ell$}
Consider the vector space $\mathrm{V} = \mathrm{Hom}(\BC^r,\BC^n)$ with the obvious $\mathrm{G} = \GL_r(\BC)$ action. Let $\theta$ be the inverse of the determinant character of $\mathrm{G}$, and let $L_\theta \to V$ be the linearized line bundle. Then $G(r,n) = V \git \mathrm{G} = V /_{\theta} \mathrm{G}$, and $L_\theta \to V$ descends to $\det(S^{\vee})$, as considered in the introduction. 

A section $s \in H^{0}(G(r,n),\det(S^{\vee})^{\otimes \ell})$ lifts to the $\mathrm{G}$-equivariant section $\tilde{s}$ of $L^{\otimes \ell}_{\theta}$. Denote by $W \subset V$ the subspace cut out by $\tilde{s}$; the hypersurface $X_{\ell} \subset G(r,n)$ is then the corresponding GIT quotient $W \git \mathrm{G}$. 

The Quot scheme $\quot_d(C,G(r,n))$ can be identified with the genus $g$ quasimap moduli space to $G(r,n)$ with parameterized component $C$ and no marked points, as described in \cite[Section 7.2]{ckm}. The subscheme $\quot_d(C,X_{\ell})$ is the quasimap moduli space to $X_{\ell} = W \git \mathrm{G}$ with parameterized component $C$, and consequently, by \cite[Theorem 7.2.2]{ckm}, it admits a perfect obstruction theory, described explicitly as follows (cf \cite[Section 4.5]{ckm}).

\begin{proposition}\label{prop:virtual_class}
    The scheme $\quot_d(C,X_{\ell})$ admits a perfect obstruction theory given by
    $$
        \left(R^\bullet \pi_{*} \left[ R\mathcal{H}om(\CE, \CF) \to \left(\det \CE^{\vee}\right)^{\otimes \ell} \right] \right)^\vee \to \tau_{[-1,0]}\mathbb{L}_{\quot_d(C,X_{\ell})}.
    $$
    Here $\CE$ and $\CF$ are the universal subsheaf and quotient over $ \quot_d(C, X_{\ell})\times C,$ pulled back from $ \quot_d(C, G(r,n))\times C$, and $\pi$ is a projection from the product to $\quot_d(C, X_{\ell})$.
\end{proposition}
\begin{proof}
Let $\rho: \quot_d(C,X_{\ell}) \to \mathsf{Bun}_{\mathsf{G}}$ be the morphism sending the subsheaf $E \subset \CO^{\oplus n}_{C}$ to its dual $E^\vee$ in the moduli stack $\mathsf{Bun}_{\mathsf{G}}$ of rank $r$ vector bundles over $C$. An analogue of Theorem 4.2 in \cite{ckm} implies that the relative obstruction theory for $\rho$ is quasi-isomorphic to $\left(R^\bullet \pi_{*} \left( F^{\bullet} \right)\right)^\vee$, where $F^{\bullet}$ is induced by $R^\bullet T_{W} = \left[T_{V}|_{W} \to L_{\theta}^{\otimes \ell} \right]$, and is explicitly given by 
\[
F^{\bullet} = \left[\mathcal{H}om(\CE, \CO^{\oplus n}) \to \left(\det \CE^{\vee}\right)^{\otimes \ell}\right].
\]
Since $\mathsf{Bun}_{\mathsf{G}}$ is a smooth Artin stack with obstruction theory given by $$\left(R^\bullet \pi_{*} R\mathcal{H}om(\CE, \CE) \right)^\vee[-1],$$ the absolute obstruction theory for $\quot_d(C,X_{\ell})$ is quasi-isomorphic to
\[
\left(R^\bullet \pi_{*} \left[ R\mathcal{H}om(\CE, \CF) \to \left(\det \CE^{\vee}\right)^{\otimes \ell}\right] \right)^\vee.
\]
\end{proof}

The virtual fundamental class $[\quot_d(C,X_{\ell})]^{vir} \in A_{e_\ell}(\quot_d(C,X_{\ell}))$ is then constructed using \cite{bf}. Here $e_{\ell} = d(n-\ell) + (1-g)(nr-r^2-1)$ is the virtual dimension of $\quot_d(C,X_{\ell}).$ The compatibility \eqref{eq:Compatibility_virtual_class} of virtual fundamental classes for the two targets $X_\ell$ and $G(r, n)$ follows formally from the compatibility of perfect obstruction theories, see \cite{kkp}. Here, for concreteness, we simply cite the result of \cite[Proposition 6.2.2]{ckm} in our setting.
\begin{proposition}\label{prop:compatibility}
    Assume $R^1 \pi_{*}\left(\det \CE^{\vee}\right)^{\otimes \ell} = 0$. Then 
    \[
    \iota_*[\quot_d(C, X_{\ell})]^{vir} = c_{top}(\BE_{\, \ell} ) \cap [\quot_d(C, G(r,n))]^{vir}.
    \]
    Note that this assumption is satisfied when $d\ell > 2g-2$.
\end{proposition}

\begin{remark}
    We note finally that the virtual class of $\quot_d(C,X_{\ell})$ and the compatibility \eqref{eq:Compatibility_virtual_class} can also be efficiently obtained from the virtual cycle $[\quot_d(C, G(r,n))]^{vir}$ via Manolache's framework in \cite{m} (see also \cite[Chapter 12]{fu}). We preferred to use the setup of \cite{ckm} in order to identify geometrically the scheme $\quot_d(C,X_{\ell})$ with a quasimap space to $X_\ell$ with fixed domain, improving the zero locus description \eqref{eqn:basicinclusion}.
\end{remark}

%Recall that $X_{\ell}=Z(s)$ cut out by a section $s\in H^{0}(G(r,n),\det \left(S^{\vee}\right)^{\otimes \ell})$. We first note that there is a surjection (not necessarily an isomorphism) 
%$$H^0(G(r,n),\text{Sym}^{\ell}(\wedge^r (\BC^n)^\vee) \otimes \CO_{G(r,n)}) \to H^{0}(G(r,n),\det \left(S^{\vee}\right)^{\otimes \ell})$$
%over $G(r,n)$ since the above map can be regraded as a morphism of $\GL_{n}(\BC)$ representation and the right hand side is an irreducible representation (see Borel-Weil-Bott theorem in \cite{wey}).  Let $\sigma$ be the element of $\text{Sym}^{\ell}(\wedge^r \BC^n)^\vee$ mapping to the section $s$. Let $\mathrm{W}\subset \mathrm{V}$ be the 

\vskip.3in

\section{Virtual intersection numbers}

\vskip.2in

\subsection{The virtual count of Theorem \ref{thm:VI-type_hypersurface}}\label{sec:proof_of_VI_formula}
We will first describe the top Chern class $c_{\text{top}}(\BE_{\, \ell})$ appearing in Proposition~\ref{prop:compatibility} explicitly, to aid in computing the virtual intersection numbers on $\quot_d(C, X_{\ell})$.

Let $\{1,\delta_1,\dots,\delta_{2g},\eta \}$ be a symplectic basis for the cohomology ring $H^*(C,\BZ).$ 
 %satisfying $$\delta_j\delta_k=\begin{cases}
%\eta & k=j+g\\
%-\eta & j=k+g\\
%0& \text{otherwise}
%\end{cases}.$$
Consider the K\"unneth decomposition of the Chern classes of the universal subsheaf $\CE$ on the product $\quot_d(C,G(r,n))\times C,$ 
\begin{equation}\label{eq:Kunneth_decomposition}
c_i(\CE^\vee)= a_i\otimes 1+\sum_{k=1}^{2g}b_i^{k}\otimes \delta_k+ f_i\otimes \eta,
\end{equation}
for $1\le i\le r$. Recall that $a_{i}=c_i\left(\CE_{p}^{\vee} \right)$ in $\eqref{eq:a_i}$.

\begin{lemma}\label{lem:Euler_class_calculation}
Assume $d\ell> 2g-2$. Then $\BE_{\ell}$ is a vector bundle of rank $d\ell-g+1$, and
\begin{equation}
c_{top}\left(\BE_{\ell}\right)= (\ell a_1)^{d\ell-g+1}e^{\frac{-\ell\phi}{a_1}} \in A^{d\ell -g +1} (\quot_d (C, G(r,n)),
\end{equation}
where $\phi = \sum_{j=1}^{g}b_1^jb_1^{j+g}$. The expression on the right is well defined since $\phi ^{g+1}=0$.
\end{lemma}
\begin{proof}
Using Serre duality and the condition $d\ell> 2g-2$, we have $$H^1(C,\left(\det E^{\vee}\right)^{\ell}) = H^0(C,\omega_{C}\otimes \left(\det E\right)^{\ell})^\vee = 0$$ for each point $E\to \CO^{\oplus n}$ in $\quot_d(C,G(r,n))$. Thus $\BE_{\ell}=\pi_* \left ( (\det \CE^\vee)^{\otimes \ell} \right )$ is a locally free sheaf of rank $\chi(C,\left(\det E^{\vee}\right)^{\ell})=d\ell-g+1$. 
    
 The first Chern class of the universal subsheaf is
    $$c_1(\det(\CE^\vee)) = c_1(\CE^\vee)= a_1\otimes 1+\sum_{k=1}^{2g}b_1^{k}\otimes \delta_k+ d\otimes \eta,$$
    and consequently the Chern character of the line bundle $\det(\CE^\vee)^{\otimes \ell}$ is
	\begin{align*}
\mathrm{ch}\left(\det(\CE^\vee)^{\otimes \ell}\right) 
&= e^{\ell(a_1\otimes 1+\sum_{k=1}^{2g}b_1^{k}\otimes \delta_k+ d\otimes \eta)}\\&=e^{\ell a_1\otimes 1}(1+\ell\sum_{k=1}^{2g}b_1^{k}\otimes \delta_k + (d \ell -\ell^2\phi)\otimes \eta).
	\end{align*}
where $\phi = \sum_{k=1}^{g}b_1^kb_1^{k+g}$. Using Grothendieck-Riemann-Roch, we obtain 
	\begin{align*}
	\mathrm{ch}(\BE_{\ell})&= \pi_*\big(\mathrm{ch}(\det(\CE^\vee)^{\otimes \ell})\cdot (1-(g-1)\eta)\big)
	\\
	&= e^{\ell a_1}(d\ell-g+1- \ell^2\phi).
	\end{align*}
	  A careful calculation further yields the total Chern class 
\[
c_t(\BE_{\ell}) = (1 + t\ell a_1)^{d\ell - g + 1} e^{-\frac{t\ell^2 \phi}{1 + t\ell a_1}},
\]
where $t$ is a formal variable recording the degree.
 The coefficient of $t^{d\ell-g+1}$ in the above expression is given by 
	\[
	\sum_{i=0}^{d\ell-g+1}(-1)^i(\ell a_1)^{d\ell-g+1-i}\cdot \frac{(\ell^2\phi)^i}{i!}=(\ell a_1)^{d\ell-g+1}e^{\frac{-\ell\phi}{a_1}}.
	\]
    In the above equality, we used the observation that $\phi^{g+1}=0$ and $d\ell-g+1\ge g$.
\end{proof}

We recall the virtual integrals involving the $b$ classes and $a$ classes over the Quot scheme $\quot_d (C, G(r, n))$, explicitly calculated in \cite[Proposition 2]{mo} using torus localization.
\begin{proposition}[\cite{mo}]\label{prop:intersection-b-classes}
For all $s\le d$, and $1\le j_1<\cdots <j_s\le g$, 
\begin{equation}
\int_{[\quot_d(C,G(r, n))]^{vir}}\big(b_1^{j_1}b_1^{j_1+g}\cdots b_1^{j_s}b_1^{j_s+g}\big)P
= \int_{[\quot_d(C,G(r,n))]^{vir}} \frac{a_1^{s}}{n^s}P.
\end{equation}
Here  $P = \prod_{k=1}^t a_{i_k}, \, \, 1\leq i_k \leq r$ is a monomial of weighted degree $e-s$. The expression evaluates to zero if $s>d$ or if superscripts of $b_1$ are repeated. 
\end{proposition}

\begin{proof}[Proof of Theorem~\ref{thm:VI-type_hypersurface}]
The top intersection numbers involving powers of $
\phi$ and a polynomial $P = \prod_{k=1}^{t} a_{i_k}$
are computed using Proposition~\ref{prop:intersection-b-classes}:
\begin{equation}\label{eq:phi_intersection}
\int_{[\quot_d(C,G(r,n))]^{vir}} \phi^{s} P = \int_{[\quot_d(C,G(r,n))]^{vir}} \frac{g!}{(g-s)!} \frac{a_1^{s}}{n^s} P,
\end{equation}
whenever \( s \leq \min\{d,g\} \). The integral is zero otherwise. The coefficient \( \frac{g!}{(g-s)!} \) counts the number of terms of the form $
b_1^{j_1}b_1^{j_1+g} \cdots b_1^{j_s}b_1^{j_s+g},$
satisfying \( 1\leq j_1 < \cdots < j_s \leq g \), in the binomial expansion of \( \phi^s \).

Using the compatibility of virtual fundamental cycles in Proposition~\ref{prop:compatibility} and the explicit formula for the top Chern class of $\BE_\ell$ in Lemma~\ref{lem:Euler_class_calculation}, we write 
\begin{align*}
\int_{[\quot_d(C,X_{\ell})]^{vir}}^{}P&=
\int_{[\quot_d(C,G(r,n))]^{vir}}^{}c_{top}(\BE_{\ell}) \cdot P\\
&= \int_{[\quot_d(C,G(r,n))]^{vir}}^{}(\ell a_1)^{d\ell-g+1}e^{\frac{-\ell \phi}{a_1}} \cdot P.
\end{align*}
Using \eqref{eq:phi_intersection}, the above integral equals
\begin{align*}
%\int_{[\quot_d(C,X_{\ell})]^{vir}}^{}P(a_1,\dots,a_r)&=
\int_{[\quot_d(C,G(r,n))]^{vir}}^{}(\ell a_1)^{d\ell-g+1}\bigg(\sum_{s=0}^{\min\{d,g\}}\frac{(-\ell)^s}{n^s}\binom{g}{s}\bigg) \cdot P.
\end{align*}
Note the condition $d\ge 2g-1$, which implies that $d\ge g$, hence the sum inside the integral becomes $(n-\ell)^g/n^g$. Thus we obtain the required formula \eqref{eqn:VI_hypersurface}.
\end{proof}

\subsection{Hypersurfaces in projective space}\label{sec:hypersurfaces_Pn} 
We now specialize the formula in Theorem~\ref{thm:VI-type_hypersurface} to count maps to a smooth hypersurface $X_\ell \subset \BP^{r}$. Let $e = d(r+1) + (1 - g)r$ and $e_{\ell} = d(r+1 - \ell) + (1 - g)(r - 1)$ denote the expected dimensions of $\quot_{d}(C, \BP^r)$ and $\quot_{d}(C, X_{\ell})$, respectively. Then \begin{align*} 
\int_{[\quot_d(C, X_\ell)]^{vir}} a_1^{e_{\ell}} &= \ell^{d\ell - g + 1} \left( \frac{r+1 - \ell}{r+1} \right)^g \int_{[\quot_d(C, \BP^{r})]^{vir}} a_1^{e} \\ &= \ell^{d\ell - g + 1} (r+1 - \ell)^g. \end{align*} The last equality follows easily from the Vafa–Intriligator formula in \eqref{eqn:VI}. Note that this gives a virtual count of maps from $C$ to $X_{\ell}$ sending $e_{\ell}$ distinct points $p_1, \dots, p_{e_\ell}$ to the intersection of $X_{\ell}$ with general hyperplanes $H_i \subset \BP^{r}$, for $1 \le i \le e_{\ell}$. 

Next we shall relate the virtual invariant for the Grassmannians $ G(r,n)$ and $G(n-r,n)$. Note that the Quot schemes $\quot_d(C, G(r,n))$ and $\quot_d(C, G(n - r, n))$ are not isomorphic outside the case $d=0$. We denote by $X_{\ell}$ and $\tilde{X}_{\ell}$ the same hypersurface viewed in $ G(r,n)$ and $G(n-r,n)$ respectively. The Vafa–Intriligator formula admits the following symmetry:
\begin{proposition}\label{prop:Duality}
    Let $\CE$ and $\tilde{\CE}$ denote the universal subsheaves of ranks $r$ and $n - r$ over $\quot_d(C, X_{\ell})$ and $\quot_d(C, \tilde{X}_{\ell})$ respectively. Then
\[
\int_{[\quot_d(C, X_{\ell}]^{\mathrm{vir}}} \prod_{k=1}^{t} c_{i_k}(\CE_p^{\vee}) = \int_{[\quot_d(C, \tilde{X}_{\ell})]^{\mathrm{vir}}} \prod_{k=1}^{t} s_{i_k}(\tilde{\CE}_p),
\]
where $c_{i_k}$ and $s_{i_k}$ denote the Chern and Segre classes, and $1 \le i_k \le r$.
\end{proposition}

\begin{proof}
Observe that in Theorem~\ref{thm:VI-type_hypersurface}, the constant factor and exponent of $a_1$ on the right hand side of \eqref{eqn:VI_hypersurface} depend only on the integers $n$, $\ell$, $g$, and $d$, and not on $r$. Using the identity $s_1(\tilde{\CE}_p) = c_1(\tilde{\CE}_p^\vee)$, we are reduced to proving the following equality of integrals:
    \[
\int_{[\quot_d(C, G(r,n))]^{\mathrm{vir}}} \prod_{k=1}^{t} c_{i_k}(\CE_p^{\vee}) = \int_{[\quot_d(C, G(n - r, n))]^{\mathrm{vir}}} \prod_{k=1}^{t} s_{i_k}(\tilde{\CE}_p).
\]
Let $S$ and $\tilde{S}$ denote the universal subbundles on $G(r, n)$ and $G(n - r, n)$, respectively. Then the cohomology classes satisfy
\[
c_i(S^\vee) = s_i(\tilde{S}) \quad \text{in } H^{2i}(G(r, n), \mathbb{Z}) \cong H^{2i}(G(n - r, n), \mathbb{Z}).
\]
 When the degree $d$ is sufficiently large, one can invoke the enumerativity of the Vafa–Intriligator formula from \cite{bertram} to complete the proof. Alternatively, the equality can be established by substituting the identities
\begin{align*}
    e_i(\zeta_1,\dots,\zeta_r) &= h_i(-\zeta_{r+1},\dots,-\zeta_n) \quad \text{for all } i \le r,\\
    J(\zeta_1,\dots,\zeta_r) &=(-1)^{r(n-r)} \tilde{J}(\zeta_{r+1},\dots,\zeta_n),
\end{align*}
into \eqref{eqn:VI}, where $\zeta_1, \dots, \zeta_n$ are distinct $n^{\text{th}}$ roots of unity. Here, $e_i$ and $h_i$ denote the elementary and complete homogeneous symmetric polynomials, respectively. To observe the second equality, we use the identity $n\zeta_i^{n-1}= \prod_{j\ne i}(\zeta_i-\zeta_j)$.
\end{proof}
\begin{proof}[Proof of Corollary~\ref{cor:VI_hypersurface_projective_space}]
We now consider $X_\ell$ as a subvariety of $\mathbb{P}^r = G(r, r+1)$. The class of a codimension $i$ plane in $X_\ell$ is given by $a_i = c_i(S^\vee)$. Dually, for $\tilde{X}_\ell \subset \mathbb{P}^r=G(1,r+1)$, the universal bundle $\tilde{\mathcal{E}}_p\to \quot_d(C,\tilde{X}_\ell)$ is a line bundle, and its Segre class satisfies
\[
s_i(\tilde{\mathcal{E}}_p) = c_1(\tilde{\mathcal{E}}_p^\vee)^i = \tilde{a}_1^i.
\]
Proposition~\ref{prop:Duality} along with  \eqref{eq:intro_VI_hypersurface_P^r} implies the virtual count stated in Corollary~\ref{cor:VI_hypersurface_projective_space}
\begin{align}
	\int_{[\quot_d(C,X_{\ell})]^{vir}}^{}\prod_{k=1}^{t}a_{i_k}
    % &=
    % %T_{d,g}(\ell)
    % \int_{[\quot_d(C,\tilde{X}_{\ell})]^{vir}}^{}\prod_{k=1}^{t}a_1^{i_k}\\
    % &
    = \int_{[\quot_d(C,\tilde{X}_{\ell})]^{vir}}^{}\tilde{a}_1^{e_{\ell}}= \ell^{d\ell - g + 1} (r+1 - \ell)^g,
	\end{align}
    where $i_1+\dots+i_t=e_\ell$. The enumerativity of the virtual count follows from Theorem~\ref{thm:enumerative}, which is proved in Section~\ref{sec:Asymptotic _Enumerativity}. 
\end{proof}

\begin{remark}\label{rem:quot_tev_discrepancy}
Let $X_\ell \subset \mathbb{P}^r = G(r, r+1)$ and consider the line $L = H_1 \cap \cdots \cap H_{r-1}$ for general hyperplanes $H_i$. Then $X_\ell$ intersects $L$ in $\ell$ distinct points. Suppose $t = e_\ell / (r - 1)$ is a positive integer. A naive virtual count of degree $d$ maps from a genus $g$ curve $C$ sending $t$ distinct points $p_1, \dots, p_t$ to general points $q_1,\dots,q_
t$ on $X_\ell$ is given by
\[
Q_{g,t}(X_\ell) := \int_{[\quot_d(C, X_\ell)]^{\mathrm{vir}}} \left( \frac{a_{r-1}}{\ell} \right)^t = \ell^{d\ell - g + 1 - t}(r + 1 - \ell)^g.
\]
Comparing this with the virtual Tevelev degree $\mathsf{Tev}_{g,t,d}$ from \cite[Theorem~1.5]{bp} (valid when $3 \le \ell \le r/2 + 1$ and $g + t \ge 2$), we find the relation
\[
Q_{g,t}(X_\ell) = \left( \frac{\ell^\ell}{\ell!} \right)^t \mathsf{Tev}_{g,t,d}(X_\ell).
\]
Note that this discrepancy does not contradict the enumerativity established in Theorem~\ref{thm:enumerative}: the Tevelev degree counts maps sending marked points on $C$ to points in $X_\ell$, whereas Corollary~\ref{cor:VI_hypersurface_projective_space} counts maps sending marked points to planes in $X_\ell$ of codimension at most $r - \ell$. This bound on the codimension is subtle and arises from the analysis of boundary contributions, which is explained in Section~\ref{sec:Asymptotic _Enumerativity}.
\end{remark}

\subsection{Virtual counts of maps to complete intersections in Grassmannians}\label{sec:complete_intersection}

We now briefly describe how to calculate the virtual count of maps from $C$ to complete intersections in $G(r,n)$.  
Let $\underline{\ell} = (\ell_1, \dots, \ell_u)$, and let $X_{\underline{\ell}}$ be a smooth complete intersection in $G(r,n)$, cut out by a general section  
\[
s \in H^0\left(G(r,n), \det(S^\vee)^{\otimes \ell_1} \oplus \cdots \oplus \det(S^\vee)^{\otimes \ell_u} \right).
\]
Assume that $d\ell_i \ge 2g - 2$ for all $1 \le i \le u$. Let $\quot_d(C, X_{\underline{\ell}})$ denote the subscheme of $\quot_d(C, G(r,n))$ defined as the zero locus $Z(\tilde{s})$ of the induced section $\tilde{s}$ of the vector bundle  
\[
\BE_{\underline{\ell}} := \pi_* \left( (\det \CE^\vee)^{\otimes \ell_1} \right) \oplus \cdots \oplus \pi_* \left( (\det \CE^\vee)^{\otimes \ell_u} \right).
\]
The virtual fundamental class of $\quot_d(C, X_{\underline{\ell}})$, and its compatibility with the virtual class of the Quot scheme $\quot_d(C, G(r,n))$, can be described by making obvious modifications to Propositions~\ref{prop:virtual_class} and \ref{prop:compatibility}.

We note that the vector bundle $\BE_{\underline{\ell}}$ has rank $\sum_{i=1}^{u}(d\ell_i - g + 1)$. The calculation in the proof of Lemma~\ref{lem:Euler_class_calculation} immediately implies
\begin{equation}
c_{\mathrm{top}}\left(\BE_{\underline{\ell}}\right) = \prod_{i=1}^{u} (\ell_i a_1)^{d\ell_i - g + 1} e^{\frac{-\ell_i \phi}{a_1}},
\end{equation}
where $\phi = \sum_{j=1}^{g} b_1^j b_1^{j+g}$. Using the compatibility of virtual fundamental classes and the explicit formula for the top Chern class, we obtain, for any monomial $P$ in $a_1, \dots, a_r$ of weighted degree $e_{\underline{\ell}}=e-\sum_{i=1}^{u}(d\ell_i-g+1)$,
\begin{align*}
\int_{[\quot_d(C, X_{\underline{\ell}})]^{\mathrm{vir}}} P &=
\int_{[\quot_d(C, G(r,n))]^{\mathrm{vir}}} c_{\mathrm{top}}(\BE_{\underline{\ell}}) \cdot P \\
&= \int_{[\quot_d(C, G(r,n))]^{\mathrm{vir}}} \prod_{i=1}^{u} (\ell_i a_1)^{d\ell_i - g + 1} e^{\frac{-\ell_i \phi}{a_1}} \cdot P.
\end{align*}
 Expanding the exponential in the last expression and using \eqref{eq:phi_intersection}, the integral becomes
\begin{align*}
\int_{[\quot_d(C, G(r,n))]^{\mathrm{vir}}} \prod_{i=1}^{u} (\ell_i a_1)^{d\ell_i - g + 1}
\left( \sum_{s=0}^{\min\{d,g\}} \frac{(-\sum_{i=1}^u \ell_i)^s}{n^s} \binom{g}{s} \right) \cdot P.
\end{align*}
Note that the condition $d \ge 2g - 1$ implies $d \ge g$, and the sum inside the integral equals $(n-\sum_{i=1}^{u}\ell_i)^{g}/n^g$. We summarize the discussion in the following theorem, which expresses the virtual intersection numbers of $\quot_d(C,X_{\underline{\ell}})$ in terms of the Vafa–Intriligator formula \eqref{eqn:VI}.

\begin{theorem}\label{Thm:VI_formula_complete_intersection}
	Assume $d \ell_i >2g-2$ for all $1\le i\le u$. For any monomial $P = \prod_{k=1}^t a_{i_k}, \, \, 1\leq i_k \leq r$ of weighted degree $e_{\underline{\ell}}=e-\sum_{i=1}^{u}(d\ell_i-g+1)$, we have
	\begin{align*}
	\int_{[\quot_d(C,X_{\underline{\ell}})]^{vir}}^{}P=
    %T_{d,g}(\ell)
    \frac{\left(n-\sum_{i=1}^{u}\ell_i\right)^g\cdot \prod_{i=1}^{u}\ell_i^{\,d\ell_i-g+1}}{n^g}\int_{[\quot_d(C,G(r,n))]^{vir}}^{}\prod_{i=1}^{u}a_1^{d\ell_i-g+1}P.
	\end{align*}   
\end{theorem}
\noindent We note the following consequence of Theorems \ref{thm:enumerative} and \ref{Thm:VI_formula_complete_intersection}.
 \begin{corollary}
    Let $X_{\underline{\ell}}\subset \BP^r$ be a complete intersection such that $\mor_d(C,X_{\underline{\ell}})$ has expected dimension for $d$ sufficiently large. There are exactly $$\prod_{i=1}^{u}\ell_i^{d\ell_i - g + 1} \left(r + 1 - \sum_{i=1}^{u}\ell_i\right)^g$$ maps from $f:C\to\BP^r$ such that $f(C)\subset X_{\underline{\ell}}$ and $f$ sends distinct points $p_1,\dots,p_t$ to codimension $i_k$ planes $H_{i_k}\subset \BP^r$ in general position for $1\le k\le t$ . Here we assume $\sum_{k=1}^{t} i_k = e_{\underline{\ell}}$ and  $1 \le i_k \le r - \sum_{i=1}^{u}\ell_i$ for all $1\le k\le t$.
\end{corollary}

\vskip.3in

\section{Asymptotic enumerativity}\label{sec:Asymptotic _Enumerativity}

\vskip.2in

We take up now the question of the enumerativity of the invariants, proving Theorem~\ref{thm:enumerative}. We recall first of all that for sufficiently large degrees, the Quot scheme $\quot_d (C, G(r,n))$ has long been known \cite{bdw, pr} to be irreducible of the expected dimension $nd - r (n-r) (g-1).$ Furthermore, in this situation, it has been known \cite{bertram} that the intersection of a top monomial in the $a$ classes is enumerative, counting the number of maps sending distinct points of the domain curve to general representatives of special Schubert cycles as prescribed by the monomial.

The enumerative interpretation of the virtual count of maps to a general hypersurface $X_\ell \subset G(r,n)$ in \eqref{eqn:VI_hypersurface} is conditional on the space of maps to $X_\ell$ having expected dimension, which is the case when $X_\ell$ is known to be weakly $g$-convex, as discussed in Remark~\ref{rem:weak_g-convexity} of the introduction. Establishing Conjecture \ref{conjecture:expected dim} would greatly enlarge the class of targets for which this count of maps is enumerative.  

To start the proof of Theorem \ref{thm:enumerative}, we assume therefore that the morphism space $\mor_d (C, X_\ell)$ has the expected dimension $$e_\ell = (n- \ell) d - (r(n-r) - 1) (g-1),\quad \text{for all} \ d\ge d_0.$$  The theorem is obtained by a variation of the argument of \cite{bertram}, as follows. 

\vskip.1in

Let $B_m \subset \quot_d (C, G(r,n))$ denote the boundary stratum parametrizing short exact sequences whose quotients have a torsion subsheaf of degree $m$. For such sequences, the subsheaf factors as 
$$0 \to E \to E_m \to \CO^n,$$ where $E_m/E$ is torsion of degree $m$ and $E_m$ is a subbundle of $\CO^n.$ We note the surjective morphisms
$$\tau_m: B_m \to \mor_{d-m} (C, G(r, n)), \, \  \{E\subset \CO^n \} \to \{E_m \subset \CO^n\},$$ 
$${\tilde{\tau}}_m: B_m \to C^{(m)} \times \mor_{d-m} (C, G(r, n)), \, \  \{E\subset \CO^n \} \to (\text{supp} \, E/E_m, \, \, E_m \subset \CO^n).$$ 
There is a universal sequence
$$0 \to \CE \to \CE_m \to \CO^n \, \, \text{on} \, \, B_m \times C,$$ and a corresponding injective map of line bundles 
$$(\det \CE_m^\vee)^\ell \to (\det \CE^\vee)^\ell \, \, \text{on} \, \, B_m \times C.$$
By pushing forward under $\pi$, we obtain the morphism
\begin{equation}
\mathbb E_{\ell, m} = R^0 \pi_* (\det \CE_m^\vee)^\ell \to \mathbb E_\ell.
\end{equation}

We assume first that $d-m \geq d_0$ and $d-m > 2(g-1).$ In this case, $\mathbb E_{\ell, m}$ is locally free and also a subbundle of $\mathbb E_\ell$. Therefore,
the section $\tilde{s}$ of \eqref{eqn:tildes} factors as
\begin{equation}
\CO \xrightarrow{\tilde{s}} \mathbb E_{\ell, m} \to \mathbb E_\ell,
\end{equation}
and we have that 
\begin{equation}
\label{eqn:b1}
\quot_d (C, X_\ell) \cap B_m \, \subset \, {\tau}_m^{-1} \left (\mor_{d-m} (C, X_\ell) \right ). 
\end{equation}

For $k$ with $1\leq k \leq t$, indexing the factors in the monomial $\prod_{k=1}^t a_{i_k},$ pick distinct points $p_k\in C$, and general special Schubert varieties $Y_k \subset G(r, n)$ representing $c_{i_k} (S^\vee).$ Importantly, as explained in 
\cite{bertram} and also used in \cite{alina}, for each $k$ there is a codimension $i_k$ subscheme $$D_{k} \subset \quot_d (C, G (r, n))$$ associated to $p_k, Y_k$ such that 
$$[D_{k}] = c_{i_k} (\CE_{p_k}^\vee),$$
and 
$$D_{k} \cap \mor_d (C, G(r,n)) = \text{ev}_{p_k}^{-1} (Y_k).$$
Here $\text{ev}_p: \mor_d (C, G(r, n)) \to G(r,n)$ is the evaluation map at $p.$ The subscheme $D_{k}$ is in fact a degeneracy locus representative of the universal Chern class $c_{i_k} (\CE_{p_k}^\vee)$ constructed using the universal morphism ${\CO^n}^\vee \to \CE_{p_k}^\vee$ on $\quot_d (C, G(r, n)).$ As explained in \cite{bertram}, we further have 
\begin{equation}
\label{eqn:b2}
D_{k} \cap B_m \, \subset \, {\tilde{\tau}}_m^{-1} \left ( p_k \times C^{(m-1)} \times \mor_{d-m} (C, G(r, n)) \right ) \cup {\tau}_m^{-1} (D_{m, k}).
\end{equation}
Here we have let $$D_{m, k} = \text{ev}_{p_k}^{-1} (Y_k) \subset \mor_{d-m} (C, G(r, n)),$$ the corresponding degeneracy locus in $\mor_{d-m} (C, G(r, n)).$
Combining \eqref{eqn:b1} and \eqref{eqn:b2}, we find 
\begin{equation}
\quot_d (C, X_\ell) \cap_{
k=1}^{t} D_{k} \cap B_m \, \, \subset \, \, \cup_{I} {\tau}_m^{-1} \left (\mor_{d -m} (C, X_\ell) \cap_{k \in I} D_{m, k} \right ),
\end{equation}
where we let $I$ index subsets of $\{1, \dots, t \}$ of cardinality $t-m.$ As 
$$i_1 + \dots + i_t = d (n -\ell)  - (r(n-r) -1 ) (g-1)\, \, \text{and} \, \, i_k \leq n -\ell -1,$$ we have 
\begin{equation}
\label{eqn:codimbound}
\text{codim} \,  \cap_{k \in I} D_{m, k} \geq d (n -\ell)  - (r(n-r) -1 ) (g-1) - m (n -\ell -1),
\end{equation}
for each $I$. The right-hand side is strictly larger than the dimension of the scheme $\mor_{d -m} (C, X_\ell)$ which is assumed expected,
$$\dim \mor_{d -m} (C, X_\ell) = (d-m) (n-\ell) - (r(n-r) -1 ) (g-1).$$ Thus the intersection $$\quot_d (C, X_\ell) \cap_{
k=1}^{t} D_{k} \cap B_m$$ is empty.

Finally, let us consider the case when $m$ is sufficiently large so that $d-m < d_0$, where $d_0$ is as before the irreducibility threshold for $\mor_d (C, X_\ell).$ The actual dimensions of the schemes $\mor_{d-m} (C, G(r, n))$ are in this case bounded. The codimension \eqref{eqn:codimbound} of the intersection $\cap_{k \in I} D_{m, k}$ is nevertheless bounded below by an expression increasing linearly with $d$. For $d$ sufficiently large therefore, this intersection is empty in $\mor_{d-m} (C, X_\ell).$

The theorem is argued similarly when $X_{\underline \ell}$ is a complete intersection in $G(r, n),$ as long as the factors $a_{i_k}$ of the integrand in Theorem~\ref{Thm:VI_formula_complete_intersection} satisfy the codimension condition $$i_k < n - \sum_{j=1}^{u} \ell_j, \, \, \text{for} \, \, 1\leq k \leq t.$$

\end{document}